\documentclass[letterpaper, 11pt]{article}
    \usepackage{amsmath, amssymb} 
    \usepackage{mathtools}
    
    \usepackage[top=1in,left=1in,right=1in,bottom=1in]{geometry}
    \usepackage[utf8x]{inputenc}
    \usepackage[T1]{fontenc}
    \usepackage{lmodern}
    \usepackage{graphicx}
    \usepackage{enumerate}
    \usepackage[dvipsnames]{xcolor}
    
    \usepackage{dirtytalk}
    \usepackage{tikz} 
    \usepackage{tikz-cd}
    \usepackage[pdftex,  colorlinks=true]{hyperref}
    \hypersetup{urlcolor=RoyalBlue,  linkcolor=RoyalBlue,  citecolor=black}
    \usepackage{cleveref}   
    \usepackage{textcomp} 

    \usepackage[calcwidth,  sc,  medium,  center,  compact]{titlesec}
    \titleformat{\paragraph}[runin]{\normalfont\normalsize\bfseries}{\theparagraph}{1em}{}

    \usepackage{amsthm}
    \usepackage{cleveref}		
    
    \usepackage{thmtools}
    \usepackage{thm-restate}    
     
     \usepackage[normalem]{ulem} 

    \newtheorem{theorem}{Theorem}[section]
    \newtheorem{lemma}[theorem]{Lemma}
    
    \newtheorem{coro}[theorem]{Corollary}
    \newtheorem{prop}[theorem]{Proposition}

    \theoremstyle{definition}
    \newtheorem{defn}[theorem]{Definition}
    \newtheorem{remark}[theorem]{Remark}

    \newtheorem*{axiom}{Axiom}
    
   \theoremstyle{remark}	
    
    \newcommand{\mf}[1]{\mathfrak{#1}}

    \newcommand{\R}{\mathbb{R}}

    \newcommand{\Z}{\mathbb{Z}}

    \newcommand{\C}{\mathcal{C}}
    
    \newcommand{\Ce}{\mathcal{E}}

    \DeclareMathOperator{\diam}{diam}
    \newcommand{\dist}{\mathsf{d}}

    \newcommand{\Sym}{\operatorname{Sym}}
    
    \renewcommand{\bar}{\overline}

     \newcommand*{\mc}[1]{\mathcal{#1}}

     \def\nest{\sqsubseteq}
     \def\trans{\pitchfork}
     
     \makeatletter
	 \def\squiggly{\bgroup \markoverwith{\textcolor{red}{\lower3.5\p@\hbox{\sixly \char58}}}\ULon}
	 \makeatother

    \newcommand{\s}{\mathfrak S}

    \usepackage{mathabx}
\newcommand{\propnest}{\sqsubsetneq}

\usepackage[shortlabels]{enumitem} \setlist{nosep}
\usepackage{ifthen}
\usepackage{setspace}
\newcommand{\showcomments}{yes}
\newsavebox{\commentbox}
\newcounter{claimcount}

\newenvironment{shortitem}{\begin{itemize}\vspace{-\parskip}\setlength{\itemsep}{1pt}
    \setlength{\parskip}{0pt}\setlength{\parsep}{0pt}}{\end{itemize}}
\newenvironment{shortenum}{\begin{enumerate}\vspace{-\parskip}\setlength{\itemsep}{1pt}
    \setlength{\parskip}{0pt}\setlength{\parsep}{0pt}}{\end{enumerate}}

\newcounter{thmcount}
\newcommand*{\numberedtheorem}[3]{\theoremstyle{plain}\newtheorem*{makethm\thethmcount}{#1}
    \ifthenelse{\equal{#2}{}}{\begin{makethm\thethmcount}#3\end{makethm\thethmcount}\stepcounter{thmcount}}
    {\begin{makethm\thethmcount}[#2]#3\end{makethm\thethmcount}\stepcounter{thmcount}}}
    
\DeclareMathOperator{\stab}{Stab}

\DeclareMathOperator{\Isom}{Isom}
\DeclareMathOperator{\OO}{O}

\newcommand{\ignore}[2]{\left\{\kern-.7ex\left\{#1\right\}\kern-.7ex\right\}_{#2}}

\DeclareMathOperator{\bigset}{Big}

\let\oldbibliography\thebibliography\renewcommand{\thebibliography}[1]{\oldbibliography{#1}\setlength{\itemsep}{-1pt}}

    \author{Harry Petyt and Davide Spriano}
    \title{Unbounded Domains in Hierarchically Hyperbolic Groups}
    \date{}
    
\begin{document}

    \maketitle  

\begin{abstract}
We investigate unbounded domains in hierarchically hyperbolic groups and obtain constraints on the possible hierarchical structures. Using these insights, we characterise the structures of virtually abelian HHGs and show that the class of HHGs is not closed under finite extensions. This provides a strong negative answer to the question of whether being an HHG is invariant under quasiisometries. Along the way, we show that infinite torsion groups are not HHGs. 

By ruling out pathological behaviours, we are able to give simpler, direct proofs of the rank-rigidity and omnibus subgroup theorems for HHGs. This involves extending our techniques so that they apply to all subgroups of HHGs.
\end{abstract}

\section{Introduction}

Hierarchically hyperbolic spaces and groups were introduced by Behrstock--Hagen--Sisto \cite{behrstockhagensisto:hierarchically:1} with the aim of providing a common framework for studying mapping class groups and cubical groups. The motivating observation was that several of the tools and techniques introduced by Masur--Minsky to study the mapping class group can be applied to (the majority of) groups acting on CAT(0) cube complexes. 

This common framework not only allows one to prove results about many groups at once, but also gives a way to transfer techniques from one setting across to others. For example, Behrstock--Hagen--Sisto \cite{behrstockhagensisto:quasiflats} proved a hierarchical version of a result of Huang \cite{huang:top} for CAT(0) cube complexes, which enabled them to deduce Farb's quasiflat conjecture for mapping class groups. Similarly, work of Bowditch on median spaces \cite{bowditch:median} has been adapted to the hierarchical setting to show that mapping class groups are semihyperbolic \cite{haettelhodapetyt:coarse}. In the other direction, uniform exponential growth is well understood for mapping class groups \cite{andersonaramayonashackleton:uniformly,mangahas:uniform}, whereas for cubical groups, this was only the case in low dimensions \cite{karsageev:uniform, guptajankiewiczng:groups} until the work of \cite{abbottngsprianoguptapetyt:hierarchically} in the setting of hierarchical hyperbolicity.

Although several aspects of hierarchically hyperbolic groups have been studied since their introduction,  such as dynamics on the boundary \cite{durhamhagensisto:boundaries},  acylindrical actions \cite{abbottbehrstockdurham:largest}, convexity properties \cite{russellsprianotran:convexity}, cubical structures \cite{hagenpetyt:projection}, and so on, the structure of the \say{generic} hierarchically hyperbolic group is far from being understood. This becomes more pronounced when considering group actions. For instance, there are well-known examples of CAT(0) cube complexes that do not admit a hierarchically hyperbolic structure, but these examples do not admit a group action. In fact, all known cocompact CAT(0) cube complexes are known to be hierarchically hyperbolic by a theorem of Hagen--Susse \cite{hagensusse:onhierarchical}. 

An extra layer of complication comes from the fact that the definition of hierarchically hyperbolic group requires more than just a geometric action on a hierarchically hyperbolic space, and it is not immediate whether the additional requirements are preserved under quasiisometries. Although it was generally agreed that this should not be the case, finding an example has proved challenging. The main reason for this is that there may be several different hierarchical structures on a group: even hyperbolic groups can be equipped with many hierarchically hyperbolic structures \cite{spriano:hyperbolic:1}. In other words, showing that a group cannot be equipped with the \say{natural} structure does not rule out the existence of exotic ones. 

The goal of this paper is to analyse the constraints on the hierarchically hyperbolic structure that stem from being a hierarchically hyperbolic group. As a consequence, we construct the first example of two quasiisometric groups where only one admits a hierarchically hyperbolic group structure.

Our main technical result is the existence of \emph{eyries}. Every hierarchically hyperbolic structure consists of a set \(\mf{S}\) of \emph{domains}, equipped with a partial order \(\nest\) that has a unique maximal element. With each domain is associated a hyperbolic space, and we say that a domain is \emph{unbounded} if its associated hyperbolic space is. A number of complications in the theory come from the fact that the \(\nest\)--maximal domain may be bounded, preventing geometric arguments that use the \(\nest\)--maximality. When looking only at the set of unbounded domains, however, we show that an analogous statement holds, namely that there is a finite collection of \(\nest\)--maximal elements, which we call \emph{eyries}. 

\numberedtheorem{Theorem~\ref{thm:eyries}}{Less-general version}
{Let $(G,\s)$ be a hierarchically hyperbolic group. Then there is a finite, $G$--invariant set $\Ce(G)=\{W_1,\dots,W_n\}\subset\s$ of pairwise orthogonal unbounded domains such that every other unbounded domain \(U \in \mf{S}\) is nested in some \(W_i\).}

\noindent\textbf{Results on the structure.}
As a first consequence of Theorem \ref{thm:eyries}, we can build upon the study of domains that are quasilines, carried out in \cite{abbottngsprianoguptapetyt:hierarchically}, to obtain the following.

\numberedtheorem{Corollary~\ref{bluff}}{}{Let $(G,\mf S)$ be a hierarchically hyperbolic group. Then \(G\) is virtually abelian if and only if each domain is either bounded or quasiisometric to a line.}

We thus obtain that there is no hierarchically hyperbolic group structure on a nonabelian free group that consists only of bounded domains and quasilines. As a measure of our understanding of the generic hierarchically hyperbolic structure, even such an innocent statement on a very well-understood group was previously unknown. 

Under an additional hypothesis called \emph{hierarchical acylindricity}, Theorem \ref{thm:eyries} was deduced from the omnibus subgroup theorem by Durham--Hagen--Sisto \cite[Cor.~9.23]{durhamhagensisto:boundaries}. There are two advantages to our approach. Firstly, by proving the theorem directly we can significantly simplify the proof. Secondly, dropping the hierarchical acylindricity hypothesis gives us restrictions that \emph{all} hierarchically hyperbolic group structures need to satisfy. For instance, the following is a simple consequence of Theorem~\ref{thm:eyries}.

\numberedtheorem{Corollary~\ref{torsion}}{}{Infinite torsion groups are not hierarchically hyperbolic groups.}

The role of torsion in proving that a group is not hierarchically hyperbolic is a leitmotif of this paper. Indeed, Theorem~\ref{thm:eyries} can be used to obtain information about the interaction between torsion elements and the  hierarchically hyperbolic group structure. This is particularly interesting because it is one of the first non-coarse tools that we have for proving that a group cannot be equipped with a hierarchically hyperbolic group structure. 

\numberedtheorem{Theorem~\ref{thm:crystallographic}}{}{A crystallographic group $G<\Isom\R^n$ admits a hierarchically hyperbolic group structure if and only if its point group embeds in $\OO_n(\Z)=\Z_2^n\rtimes\Sym(n)$.}

As well as being interesting in its own right, we see Theorem~\ref{thm:crystallographic} as a blueprint to prove that a specific group does not admit a hierarchically hyperbolic group structure. As a corollary, we obtain the first example of a virtually hierarchically hyperbolic group that is not hierarchically hyperbolic. 

\numberedtheorem{Corollary~\ref{trianglegroup}}{}{
The $(3,3,3)$ triangle group is virtually a hierarchically hyperbolic group (and, in particular, it is a hierarchically hyperbolic space), but is not a hierarchically hyperbolic group.}

Since the $(3,3,3)$ triangle group is a finite extension of the hierarchically hyperbolic group $\Z^2$, this gives a strong answer to the question of quasiisometric invariance of hierarchically hyperbolic groups: although the property of being a hierarchically hyperbolic \emph{space} is preserved by quasiisometries, not only is the property of being a hierarchically hyperbolic \emph{group} not preserved by quasiisometries, it is not even preserved by taking finite extensions. 

Moreover, since the direct product of the $(3,3,3)$ triangle group with $\Z$ acts properly cocompactly on the standard cubing of $\R^3$, it is hierarchically hyperbolic. Corollary~\ref{trianglegroup} therefore shows that direct factors of hierarchically hyperbolic groups need not be hierarchically hyperbolic.

Since the first appearance of this paper, it has become possible to deduce Theorem~\ref{thm:crystallographic} from \cite{haettelhodapetyt:coarse} and the arguments of \cite{hoda:crystallographic}, but the argument given here is considerably more direct.

\medskip
\noindent\textbf{Rank-rigidity and the omnibus subgroup theorem.}

The \emph{rank-rigidity} conjecture for groups acting on CAT(0) spaces was confirmed for CAT(0) cube complexes by Caprace--Sageev \cite{capracesageev:rank}. This result was later shown to hold for a large class of groups acting on hierarchically hyperbolic spaces, including all hierarchically hyperbolic groups, by Durham--Hagen--Sisto \cite{durhamhagensisto:boundaries}. In both cases, the proofs are quite complicated, the latter using measure theory on the \emph{HHS boundary}. In the same paper, Durham--Hagen--Sisto prove the \emph{omnibus subgroup theorem} for many groups acting on hierarchically hyperbolic spaces, including compact special groups and mapping class groups. Again, the proof is rather involved. For mapping class groups, this theorem was originally formulated by Mosher to consolidate several important theorems, including that of Birman--Lubotzky--McCarthy \cite{birmanlubotzkymccarthy:abelian} and that of Ivanov \cite{ivanov:subgroups}; Mangahas gives a lucid discussion of this in \cite{mangahas:recipe}. 

As was mentioned earlier, Theorem~\ref{thm:eyries} can be deduced from the omnibus subgroup theorem, which was previously only known to hold in the case where the group is hierarchically acylindrical. It turns out that the logic can be reversed, and Theorem~\ref{thm:eyries} can be used to prove the rank-rigidity theorem (Corollary~\ref{rankrigidity}) and the omnibus subgroup theorem (Corollary~\ref{omnibus}), allowing the hierarchical acylindricity hypothesis for the latter to be relaxed. This approach leads to proofs that use much simpler tools and are significantly shorter. (It should be noted that the results of \cite{durhamhagensisto:boundaries} hold in a slightly more general setting than that of hierarchically hyperbolic groups.) 

The version of rank-rigidity that we state is slightly sharper than the original, and it allows us to cleanly deduce the following fact, which could be obtained from results in the literature but has not previously been recorded.

\numberedtheorem{Corollary~\ref{coro:qi_invariant}}{}{If $G$ and $H$ are quasiisometric HHGs, then $G$ is acylindrically hyperbolic if and only if $H$ is.}

We remark that the omnibus subgroup theorem is a very deep result, and even a simpler proof still presents some amount of technical involvement. In particular, our proof uses a much more general version of Theorem \ref{thm:eyries}, in which we extend the result to subgroups.

\numberedtheorem{Theorem~\ref{thm:subeyries}}{}{Let \((G, \mf{S})\) be a hierarchically hyperbolic group and \(H \leq G\) be any subgroup. Then there is a finite $H$--invariant set $\Ce(H)=\{W_1,\dots,W_n\}$ of pairwise orthogonal domains with unbounded $H$--projection such that every $U$ with $\pi_U(H)$ unbounded is nested in some $W_i$. Moreover, if $H$ is finitely generated and infinite, then $\mc{E}(H)$ is nonempty.}

Obtaining a result about all subgroups of a group usually requires some amount of technical work, and  Theorem~\ref{thm:subeyries} is no exception. In particular, the subgroup considered may be heavily distorted, and it may have exotic interaction with the hierarchical structure. For this reason, the proof of Theorem~\ref{thm:subeyries} forms the most involved part of the paper, although the techniques used are, in essence, elementary. 

Apart from providing a simpler proof of the omnibus subgroup theorem, we expect more applications to stem from a better understanding of subgroups of hierarchically hyperbolic groups. For instance, Theorem~\ref{thm:subeyries} can be used to simplify part of the proof of the Tits alternative for hierarchically hyperbolic groups \cite{durhamhagensisto:correction}. It could also be useful for proving results on \emph{uniform uniform} exponential growth similar to those for mapping class groups \cite{mangahas:uniform} and square complexes \cite{karsageev:uniform}. 

\bigskip
\noindent\textbf{Acknowledgements.}

We would like to thank our respective supervisors, Mark Hagen and Alessandro Sisto, for their ongoing support and their many useful comments and suggestions. We are grateful to Jason Behrstock for helpful comments, to Jacob Russell for a useful suggestion on infinite torsion groups, and to Anthony Genevois for pointing us to \cite{clayuyanik:simultaneous}. We are particularly grateful for the care with which the referee read the manuscript, and for several comments that improved the exposition. The first author was supported by an EPSRC DTP at the University of Bristol, studentship 2123260. The second author was partially supported by the Swiss National Science Foundation (grant \#182186).

\section{Background} \label{section:background}

The definitions of \emph{hierarchically hyperbolic spaces} and \emph{hierarchically hyperbolic groups} are rather technical, and we refer the reader to \cite[Def.~1.1,~1.21]{behrstockhagensisto:hierarchically:2} for a complete account. Roughly, an HHS is a pair $(X,\s)$, where $X$ is an $E$--quasigeodesic space and $\s$ is a set, with some extra data. The important structure for us is as follows. 
\begin{itemize}
\item   For each \emph{domain} $U\in\s$ there is an $E$--hyperbolic space $\C U$ and a \emph{projection} map $\pi_U:X\to\C U$ that is coarsely onto \cite[Rem.~1.3]{behrstockhagensisto:hierarchically:2} and coarsely Lipschitz.
\item   $\s$ has a partial order $\nest$, called \emph{nesting}, and a symmetric relation $\bot$, called \emph{orthogonality}. Nest-chains are uniformly finite, and the length of the longest such chain is called the \emph{complexity} of $(X,\s)$. This also bounds the size of a set of pairwise orthogonal domains. These relations are mutually exclusive, and the complement of $\nest$ and $\bot$ is called \emph{transversality} and denoted~$\trans$.
\item   There is a bounded set $\rho^U_V\subset \C V$ whenever $U\trans V$ or $U\propnest V$. These sets, and projections of elements $x\in X$, are \emph{consistent}, in the sense that $\rho^U_W$ coarsely agrees with $\rho^V_W$ whenever $U\propnest V$ and $\rho^V_W$ is defined (this will be referred to as $\rho$--consistency), and $\min\{\dist_{\C U}(\pi_U(x),\rho^V_U),\dist_{\C V}(\pi_V(x),\rho^U_V)\}$ is bounded whenever $U\trans V$.
\end{itemize}
All coarseness in the above can be taken to be uniform \cite[Rem.~1.6]{behrstockhagensisto:hierarchically:2}. We fix a uniform constant $E$. For a domain $U$, we write $\s_U$ for the subset $\{V\in\s:V\nest U\}\subset\s$.

If $X$ is the Cayley graph of a finitely generated group $G$, then $(G,\s)$ is an HHG structure if it also satisfies the following regulating assumptions. (The exact equivariance described here is not part of the original definition, but it can always be assumed to hold by \cite[\S2.1]{durhamhagensisto:correction}.)
\begin{shortitem}
\item   $G$ acts cofinitely on $\s$, and the action preserves the three relations. 
\item   For each $g\in G$ and each $U\in\s$, there is an isometry $g:\C U\to\C gU$, and these isometries satisfy $g\cdot h=gh$.  
\item   There is equivariance of the form $g\pi_U(x)=\pi_{gU}(gx)$ and $g\rho^V_U=\rho^{gV}_{gU}$ for all $g,x\in G$ and all $U,V\in\s$ with $U\trans V$ or $V\propnest U$.
\end{shortitem}

Because the action of $G$ on $\s$ is cofinite, there are only finitely many isometry types of domains. In particular, there is a constant $B$ such that for any $U\in\s$, if $\diam\C U>B$, then $\diam\C U=\infty$. This is known as the \emph{bounded domain dichotomy}. After increasing $E$, we may assume that $E\ge B$.

More generally, we say that a group $G$ acts on an HHS $(X,\s)$ \emph{by HHS automorphisms} if $G$ acts on $X$ by isometries and satisfies the above three conditions, with the exception that the action on $\s$ need not be cofinite. Observe that if $G$ acts on the HHS $(X,\s)$ by HHS automorphisms, then the fact that the action preserves the relations on $\s$ means that we cannot have $gU\propnest U$ for any $g\in G$, $U\in\s$. Indeed, if $g$ has finite order then this would contradict the fact that $\nest$ is a partial order, and if $g$ has infinite order then this would contradict the fact that $\nest$--chains are finite.

An important feature of HHSs is that the distance between two points in the space can coarsely be recovered from their projections, using the \emph{HHS distance formula}, which we now state. For nonnegative numbers $p,q$, the quantity $\ignore{p}{q}$ is defined to be $p$ if $p\geq q$, and $0$ otherwise. It is conventional in the literature to write $\dist_U$ in place of $\dist_{\C U}$, and to suppress the map $\pi_U$ when measuring distances in $\C U$. For example, we write $\dist_U(x,\rho^V_U)$ in place of $\dist_{\C U}(\pi_U(x),\rho^V_U)$.

\begin{theorem}[Distance formula, {\cite[Thm~4.5]{behrstockhagensisto:hierarchically:2}}] \label{thm:df}
If $(X,\s)$ is an HHS, then for sufficiently large $s$ there are positive constants $A_s$ and $B_s$ such that, for any $x,y\in X$, we have 
\[-B_s+\frac{1}{A_s}\sum_{U\in\s}\ignore{\dist_U(x,y)}{s} \hspace{1mm}\leq\hspace{1mm} \dist_X(x,y) \hspace{1mm}\leq\hspace{1mm} B_s+A_s\sum_{U\in\s}\ignore{\dist_U(x,y)}{s}.\]
\end{theorem}

We now recall (the relevant aspects of) three of the axioms appearing in the definition of HHSs.

\begin{axiom}[Bounded geodesic image]
Suppose that $U,W\in\s$ satisfy $U\propnest W$. If $y,z\in X$ have $\dist_U(y,z)>E$, then $\rho^U_W$ is $E$--close to every geodesic $[\pi_W(y),\pi_W(z)]\subset\C W$.
\end{axiom}

\begin{axiom}[Partial realisation]
If $\{W_i\}$ is a set of pairwise orthogonal domains, then for any tuple $(p_i)_i$, with $p_i\in\C W_i$, there is some $x\in X$ with $\dist_{W_i}(x,p_i)\le E$ for all $i$.
\end{axiom}

\begin{axiom}[Uniqueness]
For each $\kappa$ there exists $\theta_u=\theta_u(\kappa)$ such that if $x,y\in X$ satisfy $d(x,y)>\theta_u$, then there exists $U\in \mf{S}$ such that $d_U(x,y)>\kappa$.
\end{axiom}

\begin{defn}[Standard product region]
For $U\in\s$, the standard product region associated with $U$ is the set $\mathbf{P}_U=\{x\in X: \dist_V(x,\rho^U_V)\leq E \text{ \text{whenever} } U\trans V \text{ \emph{or} } U\propnest V\}$.
\end{defn}

Standard product regions can be given a natural (coarse) product structure \cite[\S5.2]{behrstockhagensisto:hierarchically:2}. The following, which is a simplified version of \cite[Prop.~2.27]{abbottngsprianoguptapetyt:hierarchically}, is an example of this. A domain $U\in\s$ is said to be \emph{unbounded} if its associated hyperbolic space $\C U$ is unbounded.

\begin{prop}[{\cite[Prop.~2.27]{abbottngsprianoguptapetyt:hierarchically}}] \label{prop:productregion}
If $(X,\s)$ is an HHG such that there is an unbounded domain $U$ with the property that every unbounded domain is either nested in $U$ or orthogonal to $U$, then $\mathbf P_U$ is coarsely equal to $X$. If there is an unbounded domain orthogonal to $U$, then $X$ is quasiisometric to a product of unbounded HHSs.
\end{prop}

We now give a pair of definitions for use in Sections~\ref{section:consequences} and ~\ref{section:subgroups}.

\begin{defn}[Morse]
A subset $Y$ of a quasigeodesic space $X$ is said to be Morse if there is a function $M:[1,\infty)\to \R$ such that every $\lambda$--quasigeodesic in $X$ with endpoints in $Y$ stays $M(\lambda)$--close to $Y$. If $X$ is a group and $Y$ is the cyclic subgroup generated by an element $g\in X$, then we say that $g$ is a Morse element. 
\end{defn}

Morse (also known as \emph{strongly quasiconvex}) subsets have been studied by a number of authors in several contexts, with differing levels of generality; for examples, see \cite{tran:onstrongly, genevois:hyperbolicities, kim:stable}.

\begin{defn}[Big-set]
Let $g$ be an element of a group $G$ acting on an HHS $(X,\s)$ by HHS automorphisms with a fixed basepoint $x\in X$. The big-set of $g$, written $\bigset(g)$, is the set of domains $U$ for which $\pi_U(\langle g\rangle \cdot x)$ is unbounded.
\end{defn}

The next lemma is a basic statement that will play an important role in this paper.

\begin{lemma}[Passing-up lemma, {\cite[Lem.~2.5]{behrstockhagensisto:hierarchically:2}}]\label{lem:passingUpLemma}
Let $(X,\s)$ be an HHS with constant $E$. For every $C>0$ there is an integer $p(C)$ such that if $V\in\s$ and $x,y\in X$ have $\dist_{U_i}(x,y)>E$ for every element of a set $\{U_1,\dots,U_{p(C)}\}\subset\s_V$, then there is a domain $W\in\s_V$ with some $U_i\propnest W$ and $\dist_W(x,y)>C$.
\end{lemma}

We finish with a simple observation that will be used several times without reference.

\begin{lemma}[{\cite[Lem.~2.2]{behrstockhagensisto:hierarchically:2}}]
Every infinite set of domains of an HHS contains two domains that are transverse.
\end{lemma}

\section{Proof of the main theorem}

Here we prove the ``less-general version'' of the main result of the paper, Theorem~\ref{thm:eyries}. We remark that this weaker version is still good enough for many applications; from it we shall deduce most of the consequences listed in the introduction. Before stating the result, we define the \emph{transversality graph} of a subset of the index set of an HHS $(X,\s)$.

\begin{defn}[Transversality graph]
For a subset $\s'\subset\s$, the \emph{transversality graph} of $\s'$, denoted $\Gamma^\trans(\s')$, is the graph with vertex set $\s'$ and an edge joining $U,V$ whenever $U\trans V$. 
\end{defn}

Recall that a domain $U\in\s$ is said to be unbounded if its associated hyperbolic space is. In this section, we shall write $\bar\s$ to mean the set of all unbounded domains.

\begin{theorem}[Eyries] \label{thm:eyries}
Let $(G,\s)$ be an HHG. There is a finite, $G$--invariant set $\mathcal E(G)\subset\bar\s$ of pairwise orthogonal domains such that every $U\in\bar\s$ is nested in some element of $\mathcal E(G)$. We call the elements of $\mathcal E(G)$ the \emph{eyries} of $G$.
\end{theorem}

\begin{proof}
For a non-singleton component $C$ of $\Gamma^\trans(\bar\s)$, Lemma~\ref{lem:componentshaverulers} produces a domain $W_C\in\bar\s$ with the property that every vertex of $C$ is properly nested in $W_C$. In particular, $W_C$ is not a vertex of $C$. After applying Lemma~\ref{lem:componentshaverulers} a finite number of times (at most the complexity), we obtain $W\in\bar{\mf S}$ that is not transverse to any other $V\in\bar{\mf S}$. Note that this implies that, for any $g\in G$, the domain $gW$ is not transverse to or properly nested in any $V\in\bar\s$.

Repeating for each non-singleton component of $\Gamma^\trans(\bar\s)$ and taking the $\nest$--maximal domains produced gives a collection of $W_i$. The $W_i$ must be pairwise orthogonal, so there can only be finitely many of them. Moreover, $G\cdot\{W_i\}$ is finite because any infinite set of domains contains a transverse pair. As every $V\in\bar\s$ is nested in some $W_i$, this shows that the set $\Ce(G)=\bigcup_iG\cdot\{W_i\}$ has the desired properties.
\end{proof}

\begin{remark} \label{rem:nonempty}
Note that if $G$ is infinite then $\bar\s$ is nonempty by the uniqueness axiom and the bounded domain dichotomy, so the proof of Theorem~\ref{thm:eyries} shows that $\mathcal E(G)$ is nonempty in this case. In particular, Proposition~\ref{prop:productregion} tells us that every infinite HHG is quasiisometric to the standard product region associated with any one of its eyries, and any HHG with more than one eyrie is quasiisometric to a product of unbounded HHSs. 
\end{remark}

It remains to prove Lemma~\ref{lem:componentshaverulers}. We begin by making an important observation about the interaction between the relations and group action on $\s$. The results of this section and of Section~\ref{section:consequences} only depend on condition~\ref{lptd:unbounded} of Lemma~\ref{lemma:producingtransversedomains}, and only with $H=G$. The reader who is primarily interested in those results should therefore feel free to ignore the more technical conditions~\ref{lptd:b} and~\ref{lptd:c}. 

Recall that a metric space $X$ is said to be \emph{$C$--connected} if for any pair of points $x$ and $y$ there is some sequence $x=x_0,x_1,\dots,x_n=y$ such that $\dist(x_i,x_{i+1})\le C$.

\begin{lemma}[Producing transverse domains] \label{lemma:producingtransversedomains}
Let $(G,\s)$ be an HHG of complexity $c$, let $H$ be a subgroup, and let $V,W\in\s$ have either $V\trans W$ or $V\propnest W$. Suppose that any one of the following three conditions is satisfied by the pair $(H,W)$:
\begin{enumerate}[a)]
\item   $\pi_W(H)$ is unbounded; \label{lptd:unbounded} 
\item   There exists $C>10E$ such that $\pi_W(H)$ is $C$--connected and $\diam(\pi_W(H))>10^{c+1}(C+\dist_W(\rho^V_W,H))$; \label{lptd:b}
\item   There is a constant $C'>10E$, a geodesic $\gamma\subset\C W$, and $3c$ points $\{z_1,\dots,z_{3c}\}\subset H$ such that each $\pi_W(z_i)$ is $C'$--close to $\gamma$, and $\dist_W(z_i,z_j)>10C'$ for all $i\neq j$. \label{lptd:c}
\end{enumerate}
Then there is some $h\in H$ such that either $hW=W$ and $\dist_W(\rho^V_W,\rho^{hV}_W)>10E$; or $hW\trans W$ and $\dist_W(\rho^V_W,\rho^{hW}_W)>10E$.
\end{lemma}

\begin{proof}
If either one of condition~\ref{lptd:unbounded} and condition~\ref{lptd:b} is satisfied, then we can find a sequence of elements $(x_i)_{i=0}^c\subset H$ such that $\dist_W(\rho^V_W,x_0)>20E$ and $\dist_W(\rho^V_W,x_i)>10\dist_W(\rho^V_W,x_{i-1})$ for all $i>0$. If $x_ix_j^{-1}W\bot W$ for every pair $(i,j)$, then we have $x_kx_i^{-1}W\bot x_kx_i^{-1}(x_ix_j^{-1}W)=x_kx_j^{-1}W$ for all $i$, $j$, and $k$. Since the size of any set of pairwise orthogonal domains is at most $c$, there is a pair $(i,j)$ such that either $x_iW=x_jW$ or $x_ix_j^{-1}W\trans W$, because no translate of $W$ can be nested in $W$.

If $x_iW=x_jW$, then the isometry $x_ix_j^{-1}:\C W\to\C W$ sends $\pi_W(x_j)$ to $\pi_W(x_i)$. Since $\rho^{x_ix_j^{-1}V}_W=x_ix_j^{-1}\rho^V_W$ we have, perhaps after swapping $i$ and $j$, that
\[
\dist_W\left(\rho^V_W,\rho^{x_ix_j^{-1}V}_W\right) \geq\dist_W(x_i,x_j)-2\dist_W(x_i,\rho^V_W) \geq8\dist_W(x_i,\rho^V_W)>10E.
\]
In this case we can take $h=x_ix_j^{-1}$.

If $x_ix_j^{-1}W\trans W$, then by consistency for $x_i\in H$, one of two things happens. One option is that $\dist_W\big(x_i,\rho^{x_ix_j^{-1}W}_W\big)\leq E$, in which case $\dist_W\big(\rho^V_W,\rho^{x_ix_j^{-1}W}_W\big) \geq\dist_W(\rho^V_W,x_i)-E>10E$, and we can again take $h=x_ix_j^{-1}$. The other option is that $\dist_{x_ix_j^{-1}W}\big(x_i,\rho^W_{x_ix_j^{-1}W}\big)\leq E$, in which case noting that $x_jx_i^{-1}\pi_{x_ix_j^{-1}W}(x_i)=\pi_W(x_j)$ gives
\[
\dist_W\left(\rho^V_W,\rho^{x_jx_i^{-1}W}_W\right) = \dist_W\left(\rho^V_W,x_jx_i^{-1}\rho^W_{x_ix_j^{-1}W}\right) \geq\dist_W(\rho^V_W,x_j)-E>10E,
\]
and the element $h=x_jx_i^{-1}\in H$ has the desired property. 

Finally, suppose that condition~\ref{lptd:c} is satisfied. If there is no triple $(i,j,k)$ such that $z_iW=z_jW=z_kW$, then we have $|\{z_iW\}|>c$, so the elements of this set are not pairwise orthogonal. Consequently, some $z_iz_j^{-1}W$ must be transverse to $W$. Since $\dist_W(z_i,z_j)>10C'>100E$, at least one of $\pi_W(z_i)$ and $\pi_W(z_j)$ is $40E$--far from $\rho^V_W$. As in the previous cases, we can then use consistency of that point to obtain the result with $h\in\{z_iz_j^{-1},z_jz_i^{-1}\}$.

On the other hand, if there is a triple $(i,j,k)$ such that $z_iW=z_jW=z_kW$, then the isometries $z_jz_i^{-1},z_kz_i^{-1}:\C W\to\C W$ send $\pi_W(z_i)$ to $\pi_W(z_j)$ and $\pi_W(z_k)$, respectively. By the assumptions on $\{z_1,\dots,z_{3c}\}$, at least one of these isometries must move $\rho^V_W$ by more than $C'>10E$. Thus an element of $\{z_jz_i^{-1},z_kz_i^{-1}\}$ has the desired property.
\end{proof}

If we start with a pair of transverse domains, then Lemma \ref{lemma:producingtransversedomains} can be applied repeatedly to obtain an infinite sequence of pairwise transverse domains.

\begin{lemma} \label{lem:repeatedly_producing_transverse_domains}
Let $(G,\s)$ be an HHG, let $H$ be a subgroup, and let $U_0,U_1\in\s$ be transverse. Suppose that each pair $(H,U_0)$ and $(H,U_1)$ satisfies at least one of the conditions of Lemma~\ref{lemma:producingtransversedomains}. There is a sequence $(U_j)_{j\in\mathbf N}\subset H\cdot\{U_0,U_1\}$ of domains such that $U_j\trans U_{j-1}$ and $\dist_{U_j}(\rho^{U_{j-1}}_{U_j},\rho^{U_{j+1}}_{U_j})>10E$ for all $j>0$. Furthermore, if $y\in H$ has $\dist_{U_0}(y,\rho^{U_1}_{U_0})>2E$, then for any $n\in\mathbf N$ there exists $z_n\in H$ such that $\dist_{U_j}(y,z_n)>8E$ for all $j\le n$.
\end{lemma}

\begin{proof}
The sequence $(U_j)$ is produced by an inductive application of Lemma~\ref{lemma:producingtransversedomains}. Let us show that $\dist_{U_j}(y,\rho^{U_{j-1}}_{U_j})\le E$ for all $j>0$. For $j=1$ this is just consistency, because $\dist_{U_0}(y,\rho^{U_1}_{U_0})>2E$. For the inductive step, the fact that $\dist_{U_j}(\rho^{U_{j-1}}_{U_j},\rho^{U_{j+1}}_{U_j})>10E$ implies that $\dist_{U_j}(y,\rho^{U_{j+1}}_{U_j})>9E$, and consistency gives the desired inequality for $j+1$. Now, since $U_n$ is an $H$--translate of $U_0$ or $U_1$, there exists $z_n\in H$ with $\dist_{U_n}(z_n,\rho^{U_{n-1}}_{U_n})>2E$. A similar argument to the above shows that $\dist_{U_j}(z_n,\rho^{U_{j+1}}_{U_j})\le E$ for all $j<n$. Combining these inequalities for $y$ and $z_n$ yields the result.
\end{proof}

A simple consequence of Lemma~\ref{lem:repeatedly_producing_transverse_domains} is that if we have two unbounded domains that are transverse, then we can produce another unbounded domain that is strictly higher up the $\nest$--lattice. Recall that $p(C)$ denotes the integer produced by the passing-up lemma (Lemma~\ref{lem:passingUpLemma}) for constant $C$. 

\begin{prop} \label{prop:prebluff}
Let \((G, \mf{S})\) be an HHG. If \(U, V \in \bar{\mf{S}}\) are transverse, then there exists \(T\in \bar{\mf{S}}\) such that either \(U\propnest T\) or \(V\propnest T\).
\end{prop}

\begin{proof}
Let $(U_j)$ be the sequence of domains provided by Lemma~\ref{lem:repeatedly_producing_transverse_domains}. Taking $n=p(E)$, we know that there exist $y,z\in G$ such that $\dist_{U_j}(y,z)>E$ for all $j\le p(E)$. By the passing-up lemma, there is some domain $W$ satisfying $\dist_W(y,z)>E$ in which some $U_j$ is properly nested. By the bounded domain dichotomy, $W\in\bar\s$. Since $U_j$ is a translate of $U$ or $V$, there is a translate $T$ of $W$ with the desired property. 
\end{proof}

We can now prove the lemma that we used in the proof of Theorem~\ref{thm:eyries}, which is a refinement of Proposition~\ref{prop:prebluff}.

\begin{lemma} \label{lem:componentshaverulers}
For each non-singleton, connected induced subgraph $C$ of $\Gamma^\trans(\bar\s)$, there exists $W_C\in\bar{\mf S}$ with $U\propnest W_C$ for all $U\in C^0$.
\end{lemma}

\begin{proof}
Firstly, assume that \(C\) is finite. We shall prove the claim by induction on \(n = \vert C \vert\).
Suppose \(C = \{U_1,U_2\}\). By Proposition~\ref{prop:prebluff}, we can find \(T \in \bar\s\) such that one of the \(U_i\) is properly nested in $T$. We cannot have \(T \bot U_i\), for then we would have $U_1\bot U_2$. We are done if both $U_i$ are nested in $T$. Otherwise we can apply Proposition~\ref{prop:prebluff} again. By finite complexity, the process has to terminate, yielding \(W\) in which both \(U_1\) and \(U_2\) are nested. 

Now assume that the claim holds for subsets of size \(n-1\) and assume \(\vert C \vert = n\).
Since \(C\) is a finite connected graph, there is a leaf \(U_n\) of a spanning tree of \(C\), and \(C-\{U_n\}\) is still connected. By the induction hypothesis, there exists \(W\) such that every element of \(C-\{U_n\}\) is properly nested in \(W\). Note that we cannot have \(W \sqsubseteq U_n\) or \(W \bot U_n\). Hence either \(U_n \propnest W\), and we are done, or \(U_n \trans W\). In the latter case, we are done by considering \(C'= \{W, U_n\}\). 

Finally, assume that \(C\) is infinite. Let \(D\) be any finite connected subset of \(C\) of size at least 2, and let \(W_{D} \in \bar{\mf{S}}\) be a \(\sqsubseteq\)--maximal element in which all elements of \(D\) are nested. The existence of \(W_D\) is guaranteed by the previous step and finite complexity, as \(D\) is finite. Suppose that \(V \in C\) is not properly nested in \( W_D\). Consider a path \(U_0, U_1, \dots, U_n = V\) with \(U_0\in D\) and let \(U_i\) be the first vertex in the path that is not properly nested in \(W_D\). Since \(U_{i-1} \propnest W_D\) and \(U_i \trans U_{i-1}\), we need to have \(U_{i} \trans W_D\). Thus, there must be \(T \in \bar{\mf{S}}\) in which both \(U_{i}\) and \(W_D\) are properly nested, which contradicts the maximality of \(W_D\). 
\end{proof}

\section{Consequences} \label{section:consequences}

In this section we use our new-found understanding of eyries to deduce a number of structural results for HHGs. Our first two applications give an idea of how Theorem~\ref{thm:eyries} can be used to draw algebraic conclusions from the combinatorics of HHG structures.

\begin{coro} \label{torsion}
Infinite torsion groups are not HHGs.
\end{coro}

\begin{proof}
According to Theorem~\ref{thm:eyries}, if $G$ is an infinite HHG, then it acts on its finite (and nonempty) set of eyries $\mathcal E(G)$. The kernel of this action is a finite index subgroup $H$. Let $W\in\mathcal E(G)$ be an eyrie of $G$. Since $H$ fixes $W$, we have an isometric action of $H$ on $\C W$, which is an unbounded hyperbolic space. Since $\pi_W:G\to\C W$ is coarsely onto and $H$ has finite index in $G$, the action is cobounded. It follows that $H$ has an element acting loxodromically on $\C W$ \cite[\S3]{capracecornuliermonodtessera:amenable}, and this element must have infinite order.
\end{proof}

\begin{coro} \label{bluff}
The following are equivalent for an HHG $(G,\mf S)$.
\begin{itemize}
\item   $G$ is virtually abelian. 
\item   The eyries of $(G,\s)$ are all quasilines.
\item   The domains of $(G,\s)$ are all either bounded or quasilines.
\end{itemize}
Moreover, the rank of an abelian finite index subgroup of $G$ coincides with the number of eyries.  
\end{coro}

\begin{proof}
First suppose that $G$ is virtually abelian. Given an HHG structure for $G$, the kernel $K$ of the action of $G$ on $\Ce(G)$ has finite index in $G$. Since $\pi_W:K\to\C W$ is coarsely onto for each eyrie $W$, the action of $K$ on $\C W$ is cobounded. If some $\C W$ is not a quasiline, then $K$ acts coboundedly on the nonelementary hyperbolic space $\C W$, so contains a nonabelian free subsemigroup by \cite[Prop.~3.2, Lem.~3.3]{capracecornuliermonodtessera:amenable}: a contradiction. Thus the eyries of $(G,\s)$ are all quasilines. 

Now suppose that the eyries of $(G,\s)$ are all quasilines. Again, a finite index subgroup $K$ of $G$ acts coboundedly on $\C W$ for each eyrie $W$, so there is an element of $K$ acting loxodromically on $\C W$ by \cite[\S3]{capracecornuliermonodtessera:amenable}. According to \cite[Prop.~3.2]{abbottngsprianoguptapetyt:hierarchically}, this implies that every domain that is nested in an eyrie is bounded, so the domains of $(G,\s)$ are all either bounded or quasiline eyries. 

If the domains of $(G,\s)$ are all either bounded or quasilines, then the eyries are all quasilines, so the previous paragraph shows that the only unbounded domains are the eyries. The distance formula, Theorem~\ref{thm:df}, now gives a quasiisometry $G\to\Z^n$, so $G$ is virtually abelian \cite{shalom:harmonic}. 

The distance formula provides the statement about the rank. 
\end{proof}

\begin{coro}
If \((G, \mf{S})\) is an HHG that is not virtually abelian, then $G$ has an eyrie that is not a quasiline.  
\end{coro}

\subsection{Crystallographic groups} \label{subsection:crystallographic}

Recall that a \emph{crystallographic group} is a discrete subgroup of $\Isom \R^n$ that acts properly cocompactly. By Bieberbach's theorems \cite{bieberbach:bewegungsgruppen:1}, any such group $G$ fits into a short exact sequence of the form
\begin{align}  1\to H\to G\to F\to 1, \tag{$*$}\label{bieberbach} \end{align}
where $H$, called the \emph{translation subgroup}, is free abelian, and $F$, the \emph{point group}, is a finite subgroup of the orthogonal group $\OO_n(\R)$. For each $h\in H$ there is a constant $T$, called the \emph{translation length of $h$}, such that $\dist(v,hv)=T$ for all $v\in\R^n$.  By a theorem of Zassenhaus \cite{zassenhaus:algorithmus}, being crystallographic is equivalent to having a maximal abelian subgroup that is normal, free abelian, and finite index. 

\begin{theorem} \label{thm:crystallographic}
Let $G<\Isom\R^n$ be a crystallographic group. The following are equivalent.
\begin{shortenum}
\item   $G$ admits an HHG structure, \label{item:HHG}
\item   $G$ is cocompactly cubulated, \label{item:ccc}
\item   The point group of $G$ embeds in $\OO_n(\Z)=\Z_2^n\rtimes\Sym(n)$. \label{item:point}
\end{shortenum} 
\end{theorem}

\begin{proof}[Proof of Theorem~\ref{thm:crystallographic}]
The equivalence of Items \ref{item:ccc} and \ref{item:point} is given by \cite[Thm~B]{hagen:cocompactly}. If $G$ is cocompactly cubulated, then since it is virtually abelian it is an HHG by \cite[\S8]{behrstockhagensisto:hierarchically:1}. We shall now show that if $G$ admits an HHG structure then its point group embeds in $\OO_n(\Z)$.

Given any HHG structure for $G$, Corollary~\ref{bluff} shows that the eyries of $G$ are all quasilines, and that every other domain is bounded. Since $G$ is quasiisometric to $\R^n$ there must be $n$ eyries. It follows that $G$ acts on a set of $n$ pairs of points, each pair being the boundary of an eyrie. This gives a homomorphism $G\to\Z_2^n\rtimes\Sym(n)=\OO_n(\Z)$.

In the notation of \eqref{bieberbach}, we now show that $H$ is in the kernel of this map. From this, it follows that there is a well-defined induced map $F\to\OO_n(\Z)$. We then show that this induced map is injective.

Let $\Ce(G)=\{W_1,\dots,W_n\}$, and assume that $h\in H$ has $hW_i=W_j$ with $i\neq j$. Let $\lambda>1$ be a quasiisometry constant for the orbit map of $G$ on $\R^n$, and let $T$ be the translation length of $h$. Let $s>E$ be sufficiently large to apply the distance formula, Theorem~\ref{thm:df}, and let $A_s$ and $B_s$ be the associated constants.

Since $\C W_i$ is unbounded and $W_i\bot W_j$, we can use the partial realisation axiom to find $x\in G$ with the property that $\dist_{W_i}(h^{-1},x) > A_s\lambda(T+\lambda)+E+s+B_s$ and $\dist_{W_j}(1,x)\leq E$. The distance formula with threshold $s$ gives 
\begin{align*}
A_s\hspace{0.5mm}\dist_G(x,hx)+B_s &\geq \sum_\s\ignore{\dist_U(x,hx)}{s} 
    \geq\ignore{\dist_{W_j}(x,hx)}{s} \\
&\geq\ignore{\dist_{W_j}(1,hx)-E}{s}
    \geq\dist_{W_i}(h^{-1},x)-E-s \\
&> A_s\lambda(T+\lambda)+B_s,
\end{align*}
so $\dist_G(x,hx)>\lambda(T+\lambda)$, which contradicts the fact that the translation length of $h$ is $T$. Thus $H$ fixes each $W_k$. 

If some $h\in H$ acts nontrivially on the boundary of $\C W_k$ then the distance formula again leads to a contradiction with the fact that $h$ has finite translation length on $\R^n$. This establishes that $H$ is in the kernel of the map from $G$ to $\OO_n(\Z)$.

A similar argument to the above shows that if an element $f\in F$ is in the kernel of the induced map then there is a bound on the distance $f$ moves any point in $\R^n$. This holds only when $f$ is the identity, as $F$ is a subgroup of $\OO_n(\R)$.
\end{proof}

Recall that the $(3,3,3)$ triangle group is the crystallographic group defined as $T=\langle a,b,c, \vert a^2=b^2=c^2=(ab)^3=(bc)^3=(ca)^3=1\rangle$.

\begin{coro} \label{trianglegroup}
The $(3,3,3)$ triangle group $T$ is virtually an HHG but not an HHG. In particular, the group $T$ is an HHS but not an HHG.
\end{coro}

\begin{proof}
From \eqref{bieberbach}, we see that $T$ is virtually $\Z^2$, which is an HHG. By Theorem~\ref{thm:crystallographic}, if $T$ is an HHG then its point group embeds in $\Z_2^2\rtimes\Sym(2)$. However, $T$ contains an element of order 3.
\end{proof}

Corollary~\ref{trianglegroup} shows that the property of being an HHG does not pass to finite index overgroups, and in particular is not quasiisometry-invariant. This is in contrast with the property of being an HHS \cite[Prop.~1.10]{behrstockhagensisto:hierarchically:2}. It also shows that the property of being an HHG is strictly stronger than the property of admitting a proper cobounded action on an HHS. 

\begin{remark}
One can use Corollary~\ref{trianglegroup} to create acylindrically hyperbolic examples as well. The group $T*\Z$ acts properly cocompactly on the tree of flats, also known as the universal cover of the Salvetti complex of $\Z^2*\Z$. However, the HHS $T*\Z$ cannot be an HHG. Indeed, the flat-stabilising subgroup $T$ is Morse, and hence is hierarchically quasiconvex by \cite[Thm~6.3]{russellsprianotran:convexity}, so any HHG structure on $T*\Z$ would restrict to an HHG structure on $T$.
\end{remark}

\subsection{Rank-rigidity}

We now recover (a sharpened version of) the rank-rigidity result for HHGs of \cite{durhamhagensisto:boundaries}. The original proof involves doing measure theory on the \emph{HHS boundary}, but our combinatorial proof of Theorem~\ref{thm:eyries} allows us to avoid this, and is overall much simpler.

\begin{coro}[Rank-rigidity for HHGs] \label{rankrigidity}
Let $(G,\s)$ be an HHG. Then $G$ is either virtually cyclic, acylindrically hyperbolic, or has more than one eyrie. In particular, any infinite HHG either has a Morse element or is quasiisometric to a product of at least two unbounded HHSs and thus is wide in the sense of \cite{drutusapir:tree-graded}.
\end{coro}

\begin{proof}
If $G$ has no eyries then it is finite by Remark~\ref{rem:nonempty}. If $(G,\s)$ has more than one eyrie then it is quasiisometric to a product of unbounded HHSs by Proposition~\ref{prop:productregion}. Finally suppose that $G$ has a single eyrie, $\mathcal E(G)=\{W\}$. Then $(G,\s_W)$ is an HHG. By \cite[Thm~14.3]{behrstockhagensisto:hierarchically:1}, the action of $G$ on the hyperbolic space $\C W$ is acylindrical, so $G$ is either virtually cyclic or acylindrically hyperbolic. Moreover, $G$ has  a Morse element by \cite[Lem.~6.5, Thm~6.8]{dahmaniguirardelosin:hyperbolically} and \cite[Thm~1]{sisto:quasiconvexity}.
\end{proof}

Although the trichotomy presented in Corollary~\ref{rankrigidity} could be deduced from the original proof of rank-rigidity for HHGs and other theorems in the literature, to the best of the authors' knowledge this is the first time it has appeared in writing. In particular, we record the interesting consequence that acylindrical hyperbolicity is preserved by quasiisometries within the class of HHGs.

\begin{coro} \label{coro:qi_invariant}
If $G$ and $H$ are quasiisometric HHGs, then $G$ is acylindrically hyperbolic if and only if $H$ is. 
\end{coro}

\begin{remark}[Acylindrical action]
Although an HHG with more than one eyrie is not acylindrically hyperbolic, it turns out that it does admit an acylindrical action on a product of hyperbolic spaces. To see this, let $(G,\s)$ be an HHG with eyries $\mathcal E(G)=\{W_1,\dots,W_n\}$. The action of $G$ on $\mathcal E(G)$ induces an action of $G$ on the product $\prod_{i=1}^n\C W_i$, and a simple modification of the proof of \cite[Thm~14.3]{behrstockhagensisto:hierarchically:1} shows that this action is acylindrical. As the $W_i$ are pairwise orthogonal and the $\pi_{W_i}$ are coarsely onto, partial realisation ensures that this action is also cobounded.
\end{remark}

\section{Subgroups} \label{section:subgroups}

Let $H$ be a group acting on an HHS $(X,\s)$ by HHS automorphisms and fix a basepoint $x_0\in X$. Our goal for this section is to generalise Theorem~\ref{thm:eyries} so that its conclusion holds for $H$. We shall write $\bar\s{}^H$ to mean the set of domains $U\in\s$ for which the projection $\pi_U(H\cdot x_0)$ is unbounded. 

\begin{theorem}[Eyries for actions] \label{thm:subeyries}
For any group $H$ acting on an HHS $(X,\mf S)$ by HHS automorphisms there is a finite $H$--invariant set $\Ce(H)=\{W_1,\dots,W_n\}\subset\bar{\mf S}{}^H$ of pairwise orthogonal domains such that every $U\in\bar{\mf S}{}^H$ is nested in some $W_i$. We call the $W_i$ the \emph{eyries} of $H$. The set $\Ce(H)$ is nonempty if and only if $\bar\s{}^H$ is nonempty. If the image of $H$ in $\Sym(X)$ induced by the action is finitely generated and $H\cdot x_0$ is unbounded, then $\bar\s{}^H$ is nonempty, so $H$ has an eyrie. 
\end{theorem}

\begin{proof}
The proof of the existence of $\Ce(H)$ is the same as that of Theorem~\ref{thm:eyries}, except we use Lemma~\ref{lem:subcomponentshaverulers} in place of Lemma~\ref{lem:componentshaverulers}. The fact that $\Ce(H)$ is nonempty if and only if $\bar\s{}^H$ is nonempty is automatic from that proof. The criterion for $\Ce(H)$ to be nonempty is proved as Proposition~\ref{prop:thereisaneyrie}.
\end{proof}

As mentioned in the introduction, in the case where $H$ is \emph{hierarchically acylindrical}, Theorem~\ref{thm:subeyries} was proved by Durham--Hagen--Sisto as \cite[Cor.~9.23]{durhamhagensisto:boundaries}, where it is given as a corollary of the omnibus subgroup theorem. In contrast, we shall deduce the omnibus subgroup theorem from Theorem~\ref{thm:subeyries} in Section~\ref{subsection:omnibus}.

\begin{remark} \label{remark:hieromorphisms} 
In order to make the notation more comprehensible, we shall assume that $H$ is actually a subgroup of an HHG $(G,\s)$. In this case, we can take the basepoint to be $1$, so that $\bar\s{}^H$ denotes the set of domains $U$ for which $\pi_U(H)$ is unbounded. We stress that this has no effect on the arguments involved: the change is purely notational. In particular, the given proofs of Lemmas~\ref{lemma:producingtransversedomains}, \ref{lem:ensure_c_is_met}, and~\ref{lem:subcomponentshaverulers}, and of Propositions~\ref{prop:subprebluff} and \ref{prop:thereisaneyrie}, work equally well for groups acting on HHSs by HHS automorphisms. 
\end{remark}

We begin by proving a general structural result for HHSs. It shows that $\rho$--points of nested domains are ``well distributed''. One would like to say that not too many $\rho$--points can appear in a bounded set, but the $\rho$--consistency condition means that this is not true in general. This is not an issue if one considers only nest-maximal domains, however. This necessary restriction accounts for much of the apparent technicality of the statement of Lemma~\ref{lemma:rhodistribution}. 

The way we shall use the lemma will be in order to find a collection of points of $H$ that satisfy condition~\ref{lptd:c} from Lemma~\ref{lemma:producingtransversedomains}. That is, we will produce a large number of domains that are nested in a common one, and applying Lemma~\ref{lemma:rhodistribution} will show that there is a subset of these domains whose $\rho$--points are well separated. We then find a collection of points in $H$ that project close to these well-separated $\rho$--points.

Recall that $p(C)$ denotes the integer produced by the passing-up lemma (Lemma~\ref{lem:passingUpLemma}) for constant~$C$. 

\begin{lemma}[Distribution of $\rho$--points] \label{lemma:rhodistribution}
 If $(X,\s)$ is an HHS, then the following holds for all $D>50E$. Let $y,z\in X$ and $W\in\s$ have $\dist_W(y,z)\ge D$, and suppose that $U_1,\dots,U_n\in\s_W\smallsetminus W$ satisfy $\dist_{U_i}(y,z)>3E$ and have $\dist_V(y,z)<D$ for all $V$ with $U_i\propnest V\propnest W$ for some $i$. If $n\geq p(D+2E)$, then $\diam(\bigcup^n_{i=1}\rho^{U_i}_W)\ge D-30E$.
\end{lemma}

\begin{proof}
Let $U_1,\dots,U_n\in\s_W$ be as in the statement, but suppose that $\diam(\bigcup_{i=1}^n\rho^{U_i}_W) < D-30E$. By bounded geodesic image, each $\rho^{U_i}_W$ is $E$--close to a fixed geodesic $\gamma:[0,\dist_W(y,z)]\to\C W$ with $\gamma(0)=y$ and $\gamma(\dist_W(y,z))=z$. Choose $x^-,x^+\in X$ as follows. If $\dist_W(y,\rho^{U_i}_W)\leq 10E$ for some $i$ then let $x^-=y$. Otherwise, there exists $t$ such that $\dist(\gamma(t),\rho^{U_i}_W)\le10E$ for some $i$, but $\dist(\gamma(t'),\rho^{U_i}_W)>10E$ for all $i$ and all $t'<t$. Since $\pi_W$ is $E$--coarsely onto, we can choose $x^-\in X$ with $\dist_W(x^-,\gamma(t))\leq E$. Similarly define $x^+$ by considering the other end of $\gamma$. According to the assumption on the diameter of $\bigcup_{i=1}^n\rho^{U_i}_W$, we have that $\dist_W(x^-,x^+)< D$. 

Bounded geodesic image also shows that $\dist_{U_i}(y,x^-)\leq E$ and $\dist_{U_i}(z,x^+)\leq E$ for all $i$. In particular, $\dist_{U_i}(x^-,x^+)>E$. If $n\geq p(D+2E)$, then by the passing-up lemma there is some domain $V\in\s_W$ with $\dist_V(x^-,x^+)>D+2E$ and with $U_i\propnest V$ for some $i$. We have seen that $\dist_W(x^-,x^+)< D$, so $V\ne W$. By consistency we obtain $\dist_W(\rho^V_W,\rho^{U_i}_W)\leq E$, so the construction of $x^\pm$ allows us to use bounded geodesic image to find that $\dist_V(y,x^-)\leq E$ and $\dist_V(z,x^+)\leq E$. But now $\dist_V(y,z)\ge D$, which contradicts our assumptions.
\end{proof}

We now prove an analogue of Proposition~\ref{prop:prebluff}. The proof follows the same lines as that of Proposition~\ref{prop:prebluff}, but there are a number of complications that arise. Firstly, $H$ may not satisfy the bounded domain dichotomy, so the passing-up lemma may fail to yield an element of $\bar\s{}^H$, even for very large constants. We get around this problem as follows. Starting with the sequence of domains produced by Lemma~\ref{lem:repeatedly_producing_transverse_domains}, we apply the passing-up lemma for each natural number. This gives a sequence of domains where the diameters of the $H$--projections are infinite or bounded below by a divergent sequence. If any one of these domains has unbounded $H$--projection, then the proof is complete. Otherwise there are infinitely many of them, so two are transverse. 

We should then like to repeat the argument with this transverse pair in place of our starting pair of domains. Since these new domains are strictly higher up the $\nest$--lattice, this process could only be repeated finitely many times, so at some point a domain with unbounded $H$--projection must be produced. However, this is where the second difficulty comes in: $H$ may not be finitely generated, so the sets $\pi_W(H)$ need not be uniformly coarsely connected. In particular, our new transverse pair may fail all three of the conditions of Lemma~\ref{lemma:producingtransversedomains}, which would prevent us from applying Lemma~\ref{lem:repeatedly_producing_transverse_domains}. We must therefore be more careful about how we produce our new pair of transverse domains: we use Lemma~\ref{lemma:rhodistribution} to make sure that they meet condition~\ref{lptd:c} of Lemma~\ref{lemma:producingtransversedomains}.

Recall that the level of an element $U\in\s$, written $\ell(U)$, is the maximal length of a $\propnest$--chain in $\s$ with maximal element $U$.

\begin{prop} \label{prop:subprebluff}
Let $H$ be a subgroup of an HHG \((G, \mf{S})\) and let \(U, V \in \bar{\mf{S}}{}^H\) be transverse. There exists \(T\in \bar{\mf{S}}{}^H\) such that either \(U\propnest T\) or \(V\propnest T\).
\end{prop}

\begin{proof}
Let $R=\max\{\dist_G(1,\mathbf{P}_U),\dist_G(1,\mathbf P_V)\}$. We proceed inductively. For the inductive hypothesis, suppose that we have a pair of transverse domains $U^i_0\trans U^i_1$ such that 
\begin{enumerate}
\item   $\s_{U^i_k}\cap\{U,V\}$ is nonempty, for $k=0,1$; 
\item   Both pairs $(H,U^i_0)$ and $(H,U^i_1)$ satisfy at least one of the conditions of Lemma~\ref{lemma:producingtransversedomains};
\item   $\min\{\ell(U^i_0),\ell(U^i_1)\} \ge i +\min\{\ell(U),\ell(V)\}$.
\end{enumerate}
For the base case, we can take $U^0_0=U$ and $U^0_1=V$. Given $U^i_0$ and $U^i_1$, we shall either find $T$ as in the statement or we shall construct $U^{i+1}_0$ and $U^{i+1}_1$ satisfying the above.
Since $(G,\s)$ has complexity $c$, this process can only be repeated at most $c$ times, so we must eventually find the desired domain $T$. 

By the inductive hypothesis, the conditions of Lemma~\ref{lem:ensure_c_is_met} are met by $U^i_0$ and $U^i_1$. Let $D^i_1=100(ER+10E)$. During the proof we shall define a sequence $(D^i_j)$ with $D^i_j> D^i_{j-1}$. Given $D^i_j\ge D^i_1$ for some $j>0$, let $W^i_j$ be the domain produced by applying Lemma~\ref{lem:ensure_c_is_met} with constant $D^i_j$. If $H$ has unbounded projection to $\C W^i_j$, then setting $T=W^i_j$ completes the proof. Otherwise, set $D^i_{j+1}=10\diam(\pi_{W^i_j}(H))$, which is greater than $100(ER+10E)$. Since $\diam(\pi_{W^i_{j+1}}(H))>D^i_{j+1}=10\diam(\pi_{W^i_j}(H))$, we must have $W^i_j\ne W^i_k$ whenever $j\ne k$. Thus, if no $\C W^i_j$ has unbounded $H$--projection, then we eventually find a transverse pair $W^i_p\trans W^i_q$. In this case, we set $U^{i+1}_0=W^i_p$ and $U^{i+1}_1=W^i_q$. Lemma~\ref{lem:ensure_c_is_met} directly tells us that the inductive assumptions are satisfied by $U^{i+1}_0$ and $U^{i+1}_1$.
\end{proof}

The proof of Proposition~\ref{prop:subprebluff} relies on the following lemma. Recall that $p(C)$ denotes the integer produced by the passing-up lemma (Lemma~\ref{lem:passingUpLemma}) for constant~$C$.

\begin{lemma} \label{lem:ensure_c_is_met}
Let $U,V\in\s$ be transverse, and let $R=\max\{\dist_G(1,\mathbf P_U),\dist_G(1,\mathbf P_V)\}$. Suppose that we have two domains $U_0$ and $U_1$ that are transverse, have $\s_{U_j}\cap\{U,V\}\neq\varnothing$, and such that each pair $(H,U_0)$ and $(H,U_1)$ satisfies at least one of the conditions of Lemma~\ref{lemma:producingtransversedomains}. For any $D\ge 100(ER+10E)$, there is a third domain $W$ satisfying the following.
\begin{enumerate}
\item   $\diam(\pi_W(H))>D$.
\item   $\s_W\cap\{U_0,U_1\}$ is nonempty. In particular, $\s_W\cap\{U,V\}$ is nonempty.
\item   The pair $(H,W)$ satisfies condition~\ref{lptd:c} of Lemma~\ref{lemma:producingtransversedomains} with $C'=ER+10E$.
\end{enumerate}
\end{lemma}

\begin{proof}
Let $K=6cp(2D)$ and let $N=(Kp(D))^{c+1}$. Fix an element $y\in H$ with $\dist_{U_0}(y,\rho^{U_1}_{U_0})>2E$. Let $(U_j)$ be the sequence of domains produced by Lemma~\ref{lem:repeatedly_producing_transverse_domains}. Lemma~\ref{lem:repeatedly_producing_transverse_domains} also provides a point $z\in H$ with $\dist_{U_j}(z,y)>8E$ for all $j\le N$.

For any subset $A\subset\{0,\dots,N\}$ of size $p(D)$, we can apply Lemma~\ref{lem:passingUpLemma} to obtain a domain $W_A$ with $\dist_{W_A}(y,z)>D$ and with $U_a\propnest W_A$ for some $a\in A$. Since the $U_j$ are pairwise transverse, $W_A$ is not equal to any $U_j$. Choose $W_A$ to be $\nest$--minimal amongst all possibilities.

We now aim to use Lemma~\ref{lemma:rhodistribution} to find a domain $W'$ with a well-separated subset of the $U_j$ nested in it. There are two cases.

\medskip
\underline{Case 1:}
If some $W_A$ has at least $K$ of the $U_j$ nested in it, then let $B$ be the set of such indices $j$. Since $K=6cp(2D)$, it follows from Lemma~\ref{lemma:rhodistribution} that there is a set $B'\subset B$ of size $3c$ such that $\{\rho^{U_b}_{W_A} : b\in B'\}$ is $(D-40E)$--separated. Moreover, since $\dist_{U_b}(y,z)>8E$, each $\rho^{U_b}_{W_A}$ must be $E$--close to the geodesic $[y,z]\subset\C W_A$ by bounded geodesic image. Set $W'=W_A$. As the $U_j$ are translates of either $U_0$ or $U_1$, we can fix, for each $b\in B'$, an element $h_b\in H$ translating either $U_0$ or $U_1$ to $U_b$.
\hfill $\vartriangle$

\medskip
\underline{Case 2:}
Otherwise, for each $W_A$, fewer than $K$ of the $U_j$ are nested in $W_A$. Thus $|\{W_A\}|\ge\frac{N}{Kp(D)}$. Because $N=(Kp(D))^{c+1}$, we can now apply the passing-up lemma to cardinality--$p(D)$ subsets of $\{W_A\}$. Since the domain produced by the passing-up lemma has strictly higher level than any domain properly nested in it, this process can be repeated at most $c$ times. Thus, at some stage we obtain: a set $\{W_k\}$ of cardinality at least $\frac{N}{(Kp(D))^c}=Kp(D)$, all elements of which satisfy $\dist_{W_k}(y,z)>D$, and a domain $W'$ such that $\dist_{W'}(y,z)>D$ and at least $K$ of the $W_k$ are properly nested in $W'$. Otherwise, we would be able to apply the passing-up lemma again, contradicting the fact that $(G,\s)$ has complexity $c$.

We now proceed as in Case~1, but with $\{k:W_k\propnest W'\}$ in place of $B$. Lemma~\ref{lemma:rhodistribution} gives a set $B'$ of size $3c$ such that $\{\rho^{W_b}_{W'}:b\in B'\}$ is $(D-40E)$--separated, and again every $\rho^{W_b}_{W'}$ is $E$--close to the geodesic $[y,z]\subset\C W'$.

Although it may not be the case that $W_b$ is a translate of $U_0$ or $U_1$, each $W_b$ was obtained by repeatedly applying the passing-up lemma, starting with translates of $U_0$ and $U_1$. Thus, for each $b\in B'$ there is some $j_b\in\{0,\dots,N\}$ such that $U_{j_b}\nest W_b$. By $\rho$--consistency, $\dist_{W'}(\rho^{U_{j_b}}_{W'},\rho^{W_b}_{W'})\le E$. In particular, the separation of the $\rho^{W_b}_{W'}$ implies that the $U_{j_b}$ are distinct. As in Case~1, for each $b\in B'$, fix an element $h_b\in H$ translating either $U_0$ or $U_1$ to $U_{j_b}$.
\hfill $\vartriangle$

\medskip
In either case, we have a collection of $3c$ translates of $U_0$ and $U_1$ by elements $h_b\in H$, all of which are nested in some domain $W'$ with $\dist_{W'}(y,z)>D$. Because each of $U_0$ and $U_1$ has either $U$ or $V$ nested it, and because both $\dist_G(1,\mathbf{P}_U)$ and $\dist_G(1,\mathbf{P}_V)$ are at most $R$, we have that $\dist_{W'}(h_b,\rho^{h_bU_0}_{W'})\le ER+3E$ for translates of $U_0$, and similarly for translates of $U_1$. 

Fix $b_0\in B'$, and let $W=h_{b_0}^{-1}W'$. We have that $\diam(\pi_W(H))>D$; the set $\s_W\cap\{U_0,U_1\}$ is nonempty; and there are $3c$ points $h_{b_0}^{-1}h_b\in H$ whose projections to $\C W$ are $(ER+5E)$--close to a fixed geodesic and are pairwise at distance greater than $\frac D2$. The proof is complete.
\end{proof}

\begin{lemma} \label{lem:subcomponentshaverulers}
For each non-singleton, connected induced subgraph $C$ of $\Gamma^\trans(\bar\s{}^H)$, there exists $W_C\in\bar{\mf S}{}^H$ with $U\propnest W_C$ for all $U\in C^0$.
\end{lemma}

\begin{proof}
The proof is identical to that of Lemma~\ref{lem:componentshaverulers}, but with $\bar\s$ replaced by $\bar\s{}^H$, and references to Proposition~\ref{prop:prebluff} replaced by references to Proposition~\ref{prop:subprebluff}.
\end{proof}

The final proposition of this section is used to conclude the final statement of Theorem~\ref{thm:subeyries}, completing its proof. In view of Remark~\ref{remark:hieromorphisms}, we work with a finitely generated, infinite subgroup of an HHG, which is the analogue of a group acting on an HHS by HHS automorphisms, with finitely generated image in $\Sym(X)$, and with an unbounded orbit.

\begin{prop} \label{prop:thereisaneyrie}
If $H$ is a finitely generated, infinite subgroup of an HHG of complexity $c$, then there is some domain $U$ such that $\pi_U(H)$ is unbounded.
\end{prop}

\begin{proof}
Since $H$ is finitely generated and all maps $\pi_V$ are $E$--coarsely Lipschitz, there is a constant $C>10E$ such that every $\pi_V(H)$ is $C$--connected. By the uniqueness axiom, for each $n>10^{c+2}C$ there is a domain $V_n$ such that $\diam(\pi_{V_n}(H))>n$. If the set of all $V_n$ is finite, then there is some $p$ such that $\pi_{V_p}(H)$ is unbounded. Otherwise there is a pair $(p,q)$ such that $V_p\trans V_q$. Since $\diam(\pi_{V_p}(H))>10C$, consistency ensures that $\dist_{V_q}(H,\rho^{V_p}_{V_q})\leq E<C$, and similarly for $V_p$. 

Condition~\ref{lptd:b} of Lemma~\ref{lemma:producingtransversedomains} is therefore satisfied by both pairs $(H,V_p)$ and $(H,V_q)$. Lemma~\ref{lem:repeatedly_producing_transverse_domains} provides a sequence $(U_j)$ of translates of $V_p$ and $V_q$, and points $y,z_n\in H$, for each $n>10^{c+2}C$, such that $\dist_{U_j}(y,z_n)>E$ for all $j\le p(n)$. We can then apply the passing-up lemma (Lemma~\ref{lem:passingUpLemma}) to obtain domains $W_n$ with $\diam(\pi_{W_n}(H))>n$ that have level strictly greater than $\min\{\ell(V_p),\ell(V_q)\}$. If the set of $W_n$ is finite, then there is some $r$ for which $\diam(\pi_{W_r}(H))$ is unbounded. Otherwise we can find a transverse pair. We can now repeat the argument with this pair in place of $V_p$ and $V_q$. By finite complexity, this process must terminate, which completes the proof. 
\end{proof}

\subsection{The omnibus subgroup theorem} \label{subsection:omnibus}

We now prove the \emph{omnibus subgroup theorem} for HHGs. The theorem follows from Theorem~\ref{thm:subeyries} and \cite[Prop.~6.68]{genevois:hyperbolicities}, which is a variation on the general result \cite[Thm~5.1]{clayuyanik:simultaneous}, and which we now state. Recall that an isometry is said to be \emph{elliptic} if it has bounded orbits.

\begin{prop}[{\cite[Prop.~6.68]{genevois:hyperbolicities}}] \label{simultaneous}
Suppose that a group $G$ is acting on (quasigeodesic) hyperbolic spaces $X_1,\dots,X_n$, with each element of $G$ acting either elliptically or loxodromically on each $X_i$. If for each $X_i$ there is an element of $G$ acting loxodromically on $X_i$, then there is some element of $G$ that acts loxodromically on every $X_i$.
\end{prop}

\begin{coro}[Omnibus subgroup theorem] \label{omnibus}
Let $(G,\s)$ be an HHG, and let $H<G$, with eyries $\mathcal E(H)$. Assume that, for each eyrie $W$, there is an element of $H$ acting loxodromically on $\C W$. Then there is an element $h\in H$ with $\bigset(h)=\mathcal E(H)$.
\end{coro}

\begin{proof}
If $\mathcal E(H)$ is empty then this holds for any $h\in H$. Otherwise, since $\mathcal E(H)$ is finite, a finite index subgroup $H'<H$ acts trivially on $\mathcal E(H)$. In particular, $H'$ acts on $\C W$ for each $W\in\mathcal E(H)$, and if $h\in H$ acts loxodromically on $W$, then so does $h^{|H:H'|}\in H'$. By \cite[Thm~3.1]{durhamhagensisto:correction}, for any $U\in\s$, every element of $\stab_G(U)$ acts either elliptically or loxodromically on $\C U$, so we can apply Proposition~\ref{simultaneous} to see that $\mathcal E(H)\subset\bigset(h)$. Since elements of $\bigset(h)$ are pairwise orthogonal \cite[Lem.~6.7]{durhamhagensisto:boundaries}, this completes the proof.
\end{proof}

Note that the assumption that there is a loxodromic for each eyrie is weaker than the one in \cite{durhamhagensisto:boundaries}, which asks for the action of $\stab_H(U)$ on $\C U$ to factor via an acylindrical action for all domains $U$. Indeed, every unbounded acylindrical action on a hyperbolic space has a loxodromic element by \cite[Thm~1.1]{osin:acylindrically}, but not all actions containing loxodromics are acylindrical.

According to Gromov's classification of actions on hyperbolic spaces \cite{gromov:hyperbolic}, as clarified in \cite{capracecornuliermonodtessera:amenable}, if the action of $H'$ on $\C W$ does not contain a loxodromic element, then since it is unbounded, it is horocyclic. Note that it is difficult to make general statements about horocyclic actions, because every group admits a horocyclic action on a hyperbolic space, namely its combinatorial horoball. If the group is countable and discrete then this action is proper.
 
We conclude that if $H$ does not satisfy the conditions of Theorem~\ref{omnibus}, then there is some $W\in\mathcal E(H)$ such that the following hold for any $x\in\C W$.
\begin{shortitem}
\item   $H'\cdot x$ is unbounded.
\item   $\langle h\rangle\cdot x$ is bounded for all $h\in H'$.
\item   $H'\cdot x$ is not quasiconvex (in particular, it is not coarsely dense) \cite[Prop.~3.2]{capracecornuliermonodtessera:amenable}.
\item   $H'$ has a unique limit point in $\partial\C W$, and this point is the unique finite orbit of $H'$ in $\partial\C W$ \cite[Prop.~3.1]{capracecornuliermonodtessera:amenable}.
\end{shortitem}
\small\singlespacing
\bibliographystyle{alpha}
\bibliography{bibtex}
\end{document}